\documentclass[12pt]{amsart}
\usepackage{amssymb}
\usepackage[latin1]{inputenc}
\usepackage{amsthm}
\usepackage{amsfonts}
\usepackage{mathrsfs}
\usepackage{graphicx}
\usepackage{amsmath}
\usepackage[all]{xy}

\let\mathcal\mathscr
\def\bA{{\mathbb A}}

\newtheorem{thm}{Theorem}[section]

\def\P{{\bf P}}

\def\R{{\bf R}}

\def\C{{\bf C}}

\def\Pic{\mathop{\rm Pic}\nolimits}

\def\Spec{\mathop{\rm Spec}\nolimits}

\def\tilde{\widetilde}

\def\phi{\varphi}

\def\cM{{\mathcal M}}
\def\cO{{\mathcal O}}

\def\div{\mathop{\rm div}\nolimits}

\numberwithin{equation}{section}

\newtheorem{theorem}[thm]{Theorem}

\newtheorem{conjecture}[thm]{Conjecture}
\newtheorem{corollary}[thm]{Corollary}

\newtheorem{example}[thm]{Example}

\newtheorem{lemma}[thm]{Lemma}

\newtheorem{proposition}[thm]{Proposition}

\newtheorem{remark}[thm]{Remark}

\newtheorem{Claim}[thm]{Claim}
\newtheorem{Fact}[thm]{Fact}

\begin{document}

\title {On some differences between number fields and function fields}
\author{Carlo Gasbarri}\address{Carlo Gasbarri, IRMA, UMR 7501
 7 rue Ren\'e-Descartes
 67084 Strasbourg Cedex}

\date{}
\thanks{Research supported by the FRIAS-USIAS}


\begin{abstract}
The analogy between the arithmetic of varieties over number fields and the arithmetic of varieties over function fields is a leading theme in arithmetic geometry. This analogy is very powerful but there are some gaps. In this note we will show how the presence of isotrivial varieties over function fields (the analogous of which do not seems to exist over number fields) breaks this analogy. Some counterexamples to a statement similar to Northcott Theorem are proposed. In positive characteristic, some explicit counterexamples to statements similar to Lang and Vojta conjectures are given. 
\end{abstract}

\maketitle

\tableofcontents
\section {Introduction}

\

Since the XIX century an analogy between the arithmetic of a number field and the arithmetic of a field of rational functions of an algebraic curve has been observed. For instance both are fields of fractions of suitable Dedekind domains where a so called product formula holds. This kind of fields is nowadays called a "global field". We expect that the arithmetic theory of the algebraic points of algebraic varieties over global fields may have similar features, thus a similar theory. 

More concretely one expects that there should exist a "formal language" with many models. Some of these models are builded up from the varieties over number fields and others are builded up from the varieties over function fields. A statement proved int this language will give then theorems in both theories. 

Ideas f gave many interesting applications: for instance the description of the class field theory using ad\`eles and id\`eles is one of the big achievements   of this. 

The theory of schemes in algebraic geometry also provides a good example of language which can be applied both over function fields and over number fields. Moreover, Arakelov theory push forward this analogy to obtain a good intersection theory which, with some caveat, is formally the same. 

At the moment the language of the analogy is sufficiently developed in order to allow to formulate common conjectures and ideas. Lang and Vojta conjectures are leading ideas in this contest. Over a number field, the Lang conjecture predicts that the rational points of a variety of general type should be not Zariski dense. Over a field of functions in characteristic zero, an analogous conjecture can be stated but one has to exclude varieties which, after a field extension, are birational to varieties  defined over the base field (cf. after). One of the aims of this note is to show that, for function fields in positive characteristic, even a weak form of this is false. 

Usually, when one wants to prove a theorem on the arithmetic of rational (algebraic) points of varieties over global fields, the situation is more favorable in the function fields case. This is principally due to the fact that, over these fields, an horizontal derivation is available (there is a non trivial derivation over the base field). This is why many statements which are conjectural over varieties over number fields are proved in the analogous situation over function fields.  Consequently, it is widely believed that a conjecture in this theory should be checked before over function fields and then, once the proof is well understood there, one should try to attack it for varieties over number fields. We want to show, mainly by examples, that some part of height theory seems to better behave over number fields then over function fields. This, again,  is due to the existence of the so called isotrivial varieties (the analogous of which do not seem to exist over number fields). 

In the last part of this paper we will construct explicit examples of surfaces over a function field of positive characteristic which are of general type, are not birational to isotrivial surfaces and which are dominated by a surface defined over the base field. These surfaces will provide counterexamples to statements similar to Lang and Vojta conjectures. 

The fact that part the analogy is broken by the existence of isotrivial varieties is, in our opinion, a very important issue which should be analyzed more deeply. A better comprehension of it would probably improve aspects of the analogy and will lead to a development of the common language. This will allow to perhaps better formulate the leading conjectures of the theory. 

This note is based on the talk I gave at the "Terzo incontro italiano di teoria dei numeri" held in Pisa in september 2015. I would warmly thank the organizers, in particolar Andrea Bandini and Ilaria Del Corso for the perfect organization of the meeting and for the possibility they gave to me to give a talk. 

\section{Notations, terminology}

In the sequel $K$ will be a global field. Thus $K$ may be either a number field or the field of rational function of a smooth projective curve $B$ over the complex numbers or the field of rational functions of a smooth projective curve over an algebraically closed field $k$ of positive characteristic. When the base field $K$ is a number field we will say that "we are in the number field case, otherwise we will say that we are dealing with the "function fields case". 

In both situations we will denote by $\overline{K}$ the algebraic closure of $K$.

If $L/K$ is a finite extension. In the function field case, there is a unique smooth projective curve $B_L$ with a finite morphism $\alpha: B_L\to B$.  If we denote by $g_L$ the genus of $B_L$,  we will denote by $d_L$  the number ${{2g_L-2}\over{\deg(\alpha)}}$.  In the number field case, by analogy with the above,  we will denote by  $d_L$ the logarithm of the absolute value of the relative discriminant of $L$ over $K$. 

We suppose now that we are in the function field case. In this case we will denote by $k$ the field $\C$ or the aforementioned field $k$. 

Let $X_K$ be a smooth projective $K$--variety. By a model of $X_K$ over $B$ we mean a normal projective $k$-variety $X$ (even smooth when $k$ is $\C$) with a flat projective morphism $p:X\to B$ such that the following diagram is cartesian

$$\xymatrix{X_K\ar[r] \ar[d] & X \ar[d]^p \\
\Spec(K)\ar[r] & B.\\}$$

It is very easy to construct models of $X_K$: a model of it may be realized as a closed set of $\P^N\times B$. Such a model, in general, won't  be regular and not even normal. If we consider the normalization of it (and, in characteristic zero, resolution of singularities of it) one may always construct normal projective models of $X_K$ (and even smooth, in characteristic zero). 

If $H_K$ is a line bundle over $X_K$, by  a model of $H_K$ over $B$ we mean a couple $(X,H)$ where $X$ is a model of $X_K$ over $B$ and $H$ is a line bundle over $X$ whose restriction to $X_K$ is $H_K$. Since every line bundle is difference of very ample line bundles, models of $(X,H)$ always exist.

Suppose that $p\in X_K(L)$ is a $L$-rational point and $X$ is a model of $X_K$ over $B$. By the valuative criterion of properness, there is an unique $k$--morphism $P:B_L\to X$ such that $p\circ P=\alpha$ and the following diagram is cartesian
$$\xymatrix{\Spec(L)\ar[r] \ar[d] ^p& B_L \ar[d]^P \\
X_K\ar[r] & X.\\}$$

We will say that $P$ is the model of the point $p$ over $X$.

Suppose that $X_K$ is a variety. We will say that $X_F$ is {\it isotrivial} if we can find a variety $X_0$ {\it defined over $k$} and an isomorphism $X_K\times_K\Spec(\overline{K})\simeq X_0\times_kSpec(\overline{K})$.

For instance, the projective space $\P_N$ is isotrivial (and the isomorphism may be defined over $K$). If $K=k(t)$ and $X_K$ is the curve $\{y^2z=x^3+tz^3\}\subset\P^2$; then $X_K$ is isotrivial but it is not defined over $k$: It will be isomorphic to $y^2z=x^3+z^3$ over the field $k(t^{1/6})$.

Suppose that $X_K$ is a smooth variety, let $f:X\to B$ be a model of it. If we restrict $f$ to an open set $U$ of $B$, we may suppose that the morphism $f$ is smooth, The restriction to the generic fibre of the canonical exact sequence of differentials associated to $f$ give rise to an extension
\begin{equation}\label{ks}
0\longrightarrow \cO_X\longrightarrow E\longrightarrow \Omega^1_{X_K/K}\longrightarrow 0
\end{equation}
which gives a class $KS(X_K)\in H^1(X_K,(\Omega^1_{X_K/K})^\vee)$, called the {\it Kodaira Spencer class of $X_K$}. It is independent on the model $X$. The following  important fact holds:

\begin{Fact}\label{isotrivial2} Let $X_K$ be a smooth variety over a function field (of any characteristic). If the Kodaira Spencer class of $X_K$ is non zero, then $X_K$ is not isotrivial. 
\end{Fact}

Let's sketch why Fact \ref{isotrivial2} holds: suppose that there exists a smooth projective variety $X_0$ defined over $k$ such that $X_0\times_kK\simeq X_K$ (the isomorphism is defined over $K$), then one easily sees that $X_0\times B$ is a model of $X_K$ and  the exact sequence \ref{ks} is split. 
If $K'/K$ is a finite extension, denote by $X'$ the $K'$ variety $X_K\times_KK'$.  One easily checks that  one has an isomorphism 
$H^1(X_K,(\Omega^1_{X_K/K})^\vee)\otimes K'\simeq H^1(X_{K'}(\Omega^1_{X_{K'}/K'})^\vee)$ and  the image of $KS(X_K)\otimes 1$ via this isomorphism  is $KS(X_{K'})$.  Thus, if  there exists a finite extension $K'/K$ and an isomorphism $X_0\times_kK'\simeq X_0\times_KK'$ then $KS(X_K)\otimes 1=0$ and consequently $KS(X_K)=0$. 

\smallskip

One of the leading conjecture on arithmetic of varieties over global fields is the Lang conjecture: We recall that if $X$ is a smooth projective variety defined over a field and $K_X$ is the canonical bundle of it, then $X_K$ is said to be {\it of general type} if $h^0(X_K, K_{X}^n)\sim_n n^{\dim(X)}$.

\begin{conjecture} (Lang) Let $K$ be a global field of characteristic zero and $X_K$ be a smooth projective variety of general type defined over $K$. If $K$ is a function fiel, then we also suppose that $X_K$ is not birational to an isotrivial variety. Then $X(K)$ is not Zariski dense.
\end{conjecture}

In the last section of this paper we will show that the hypothesis on the characteristic of the field is necessary.

\section{Height theories and remarks on Northcott theorem}

\

Suppose that $K$ is a global field as before. If $X_K$ is a projective variety, we we denote by $FUB(X_K)$ the group of functions $f:X_K(\overline{K})\to \R$ up to bounded functions. 

The main properties of height theory for varieties over number fields may be  resumed by the following statements:

Suppose that $K$ is a number field. There is a unique map of groups
$$\xymatrix{h:Pic(X_K)&\ar[r]&FUB(X_K)\\
L&\ar[r] &h_L(\cdot)\\}$$
(we will say that $h_L(\cdot)$ is the height associated to $L$). 
such that:

\smallskip

(i) It is functorial in $X_F$: if $\varphi: X_F\to Y_F$ is a morphism of varieties, then, for every $L\in Pic(Y_K)$ and every $p\in X_K(\overline{K})$ we have $h_L(\varphi(p))=h_{\varphi^\ast(L)}(p)$.

\smallskip

(ii)  If $X_F$ is the projective space $\P_N$ and $L=\cO(1)$ then the standard Weil height is in the class of $h_L(\cdot)$ .

Moreover the following properties are verified:

\smallskip

a) If $D$ is an effective divisor on $X_K$ and $L=\cO_{X_K}(D)$, then $h_L\geq O(1)$ on $(X_K\setminus D)(\overline{K})$.

\smallskip

b) (Northcott Theorem) Let $L_K$ be an ample line bundle over $X_K$ and let $h_L(\cdot)$ be a function representing the height with respect to $L_K$. Suppose that $A$ and $B$ are positive constants. Then the set 
$$\left\{ p\in X_K(\overline{K})\,\, {\rm s.t.} [K(p):K]\leq B \;\;{\rm and}\;h_L(p)\leq A\right\}$$
is finite.

When $K$ is a function field, a theory formally similar to height theory is available:

Suppose now that $K$ is a function field. There is a unique map of groups $h_L:Pic(X_K)\to FUB(X_K)$ which verify property (i) above and which verify the following

(ii') If $X_F$ is the projective space $\P_N$ and $L=\cO(1)$  then class $h_L$ is computed as follows:
Suppose that $p\in X_K(L)$ and $P:B_L\to\P^N$ is the associated morphism; then 
$$h_L(p)={{\deg(P^\ast(L))}\over{[L:K]}}.$$

It is easy to verify that a property similar to property (a) above holds in this case. Moreover the proof of this is formally the same in the function field and in the number field case. 

On the opposed side, property (b) above fails in general. 

We will now describe some examples which show the failure of Northcott property of heights over function fields. 

\begin{example} Suppose that $X_K=\P_N$ and $L=\cO(1)$. Then, every point $p\in X_K(k)$ give rise to a point $p\in X_K(K)$ and it is easy to see that all these points have bounded height (he bound will depend on the model of $X_K$ we choose). Moreover these points are Zariski dense. 
\end{example}

Of course one may object that the example above is isotrivial. But it is not easy to change it in a non isotrivial example:

\begin{example} Fix $r>N+4$ non trivial morphisms $f_i:B\to\P_N$. We suppose that the morphisms $f_i$ are not conjugate under the action of $PGL(N+1)$. Each one of the $f_i$'s defines a point $p_i\in P_N(K)$. None of this point is a point of $\P_N(k)$. Let $X_K$ be the blow up of $\P_N$ in these points. Then:

1)  $X_K$ is not isotrivial;

2) The set of $K$--rational points of bounded height (with respect to an ample line bundle) is Zariski dense.
\end{example}

Let's explain why (1) and (2) of example above hold.

1) Let  $X_K$ be a a $K$--variety and i $X\to B$ a model of of it. For every closed point $b\in B(k)$  we denote by $X_b$ the fiber of $X$ over $b$; it is a projective $k$-variety. The variety $X_K$ is isotrivial if and only if there is a non empty open set $U\subseteq B$ such that, for every $b\in U(k)$, the variety $X_b$ is $k$--isomorphic to a fixed $k$-variety $X_0$. We observe the following fact: if $X_1$ is the variety obtained by blowing up $\P_N$ in $N+4$ points in general position  and $X_2$ is  the variety obtained by blowing up $\P_N$ in another $N+4$--uple of  points  in general position (which is not  in the $PGL(N+1)$--orbit of the previous one), then $X_1$ and $X_2$ are {\it not} isomorphic (one easily sees that they can be isomorphic if and only if the blown up points are in the same orbit under $PGL(N+1)$).

Consequently the $f_i$'s are not conjugate under $PGL(N+1)$ and $b$ and $b'$ are two general points of $B$,  then the sets $\{ f_i(b)\}$ and $\{f_i(b')\}$ are not conjugate under the same action. Thus the corresponding $X_b$ and $X_{b'}$ are not isomorphic. Thus $X_K$ is not isotrivial.

\smallskip

2) Each point $q\in\P_N(k)$ rises to a point  $q_1$ of $X_K$. Denote by $\pi:X_K\to\P_N$ the projection and by $L$ the line bundle $\pi^\ast(\cO(1))$. By functoriality, we have that for each point $q_1$ as above, we have that $h_L(q_1)$ is bounded independently of $q_1$. Moreover $L$ is a big bundle, thus, we can find an effective divisor $D$ and an ample divisor $A$ on $X_K$ such that, for $n$ sufficiently big we have $nL=A+D$. By property (a) of heights, the height  with respect to $A$ of the points $q_1$ as above which are not in $D$ is bounded from above independently on $q_1$.

\smallskip

The main criticism we can do to the example above is that the variety $X_K$ is {\it birational} to an isotrivial variety. If we focus our attention, not on rational points, but on points of bounded degree, even this objection can be dramatically abandoned. 

\smallskip

\begin{example} Let $X_K$ be {\rm any} curve defined over $K$ (isotrivial or not). Let $f:X_K\to\P_1$ be a morphism defined over $K$. Let $L_K$ be the line bundle $f^\ast(\cO(1)$ (it is an ample line bundle over $X_K$). Let $d=\deg(f)$. Fix a representative of $h_L(\cdot)$. Then we can find a constant $A$ such that the set 
$$\{p\in X_L(\overline{K}) \;/\; [K(p):K]\leq d \;\; {\rm and}\;\; h_L(p)\leq A\}$$
is infinite (thus Zariski dense). 
\end{example}

Indeed, if we take a point $p\in X_K(\overline{K})$ such that $f(p)\in \P_1(k)$ then, by functoriality of the heights, we have that $h_L(p)\leq A$ for a suitable constant $A$ independent on $p$.
Such a $p$ is defined over an extension of $K$ which is of degree less or equal then $d$ (because, in particular $f(p)\in\P_1(K)$). 

We remark that the example above may even be strengthened in chacteristic zero (or when $d$ is coprime to the characteristic of $k$) : a refinement of the argument above gives that there is a constant $B>d_K$ such that 
$$\{p\in X_L(\overline{K}) \;/\; d_{K(p)}\leq B \;\; {\rm and}\;\; h_L(p)\leq A\}$$
is Zariski dense.

Indeed,  extend the morphism $f$ to a morphism $F$ from a model $X$ of $X_K$ to $\P_1\times B$. 

Let $R$ be the branch divisor of $F$.  Let $b$ be the degree of $R$ over $\P_1$. If $p\in X_K(\overline{K})$ is a point such that $f(p)\in\P_1(k)$ then the curve $B_{K(p)}$ is a covering of $B$ of degree at most $d$ and ramified in at most over $b$ points. Thus, by Hurwitz formula, the genus of 
$B_{K(p)}$ is bounded independently of $p$. 

\smallskip

At the moment the best result we know in the direction of an analogous of Northcott theorem in the function fields case is the following theorem due to Moriwaki \cite{Mo1}:

\begin{theorem} \label{moriwaki} Let $X_K$ be a projective variety which is, either of general type or it do not contain any rational curve. Let $L_K$ be an ample line bundle over $X_K$ and $h_{L_K}(\cdot)$ be a representative of the height with respect of it. Let $A$ be a constant.  Suppose that the set 
$$\{ p\in X_K(K)\;\; h_{L_K}(p)\leq A\}$$
is Zariski dense on $X_K$.

Then $X_K$ is birational to an isotrivial variety. 
\end{theorem}

Of course this theorem very well apply to curves, abelian varieties, geometrically hyperbolic varieties etc. but in our opinion it should be generalized and we should find the most general statement. For instance a statement which is true for varieties of arbitrary Kodaira dimension. 

A refinement of the Lang conjecture above is the more ambitious Vojta conjecture: 

\begin{conjecture} (Vojta) Suppose that $K$ is a global field of characteristic zero, $X_K$ is a smooth projective variety defined over it and $K_X$ be the canonical line bundle of $X_K$. Then we can find a proper closed subset $Z\varsubsetneq X_K$ and a positive constant $A$ such that, for every $p\in X_K(\overline{K})\setminus Z$ we have
\begin{equation}
h_{K_X}(p)\leq A\cdot d_{K(p)}+O(1)
\end{equation}
\end{conjecture}

Remark that Vojta conjecture above implies Lang conjecture {\it only in the number fields case}. In the Function field case it implies some kind of arithmetic statement only if we can couple it with theorem \ref{moriwaki}. It is known for curves, cf. for instance \cite{Ga1} where a stronger version of it holds. This version have been proved by Yamanoi and McQuillan (independently). Vojta conjecture holds also  for varieties with ample cotangent bundle \cite{Mo1} and for  a  big class of surfaces \cite{Mq}. In positive characteristic, it is false, we show some counterexamples in the next section. Nevertheless one can see \cite{Ki} for the case of curves in positive characteristic.

\

\section{explicit counterexamples in positive characteristic}

\

In this section we show that, if $K$ is a function field in positive characteristic, we can always explicit examples of varieties of general type which are non isotrivial and having the set of $K$--rational points which is Zariski dense. We will also show that in some explicit examples, the set of rational points with bounded height is Zariski dense. Thus the Lang conjecture is false in this case and its statement should be corrected. 

Let $K$ be a field of positive characteristic $p>2$ (algebraically closed in the first part of this section) and let $X$ be a smooth projective variety defined over it. Let $L$ be an ample line bundle over $X$. We fix a Zariski covering $\{ U_i=\Spec(A_i)\}_{i\in I}$  by affine open sets of $X$ and a cocycle $\{ g_{ij}\}$ submitted to it and defining $L$. 

Let $s\in H^0(X;L^p)$ be a non zero section. We may suppose that it  is locally defined by functions $f_i\in A_{i}$ submitted to the conditions $f_i=g^p_{ij}f_j$ on $U_i\cap U_j$.

We associate to $s$ an inseparable covering of $X$ as follows: We consider the schemes $\Spec(A_i[z_i]/(z_i^p-f_i))$ glued together over $U_i\cap U_j$ by $z_i=g_{ij}z_j$. This give rise to a scheme $Z_s$ with a finite, totally inseparable morphism $f_s:Z_s\to X$. we will call $Z_s$ {\it the  inseparable ramified  $p$--covering associated to $s$}.

Remark that the morphism $f_s$ is actually ramified everywhere, but the name is chosen in analogy with the prime to $p$ case. 

The section $s$ defines a global differential $d(s)\in H^0(X;\Omega^1_{X/F}\otimes L^p)$ as follows:

Locally, over $U_i$ we define $d(s)|_{U_i}:=d(f_i)$. Since $f_i=g_{ij}^pf_j$ we have that $d(f_i)=g_{ij}^pd(f_j)$ over $U_i\cap U_j$. Thus the $d(f_i)$ glue to a global form $d(s)\in H^0(X;\Omega^1_{X/F}\otimes L^p)$.

\subsection{Regularity of $Z_s$}

Let $z\in Z_s$ be a closed point and $x=f_s(z)$. Choose, over an algebraic closure $\overline{K}$ of $K$, an isomorphism between the completion $\widehat{\cO_{X,x}}$ of the local ring of $X$ at $x$ with the ring $\overline{K}[\![ x_1,\dots, x_n]\!]$.  The restriction of $d(s)$ to $\widehat{\cO_{X,x}}$ may be written as $h_1d(x_1)+\dots +h_nd(x_n)$.

\begin{Claim} The point $z$ is singular if and only if the ideal $(h_1,\dots, h_n)$ is contained in the maximal ideal of $\widehat{\cO_{X,x}}$.
\end{Claim}
\begin{proof} The regularity of $Z_s$ may be checked on the completions. Choose $i$ such that $x\in U_i$. Then the restriction of $Z_s$ to $Spf(\widehat{\cO_{X,x}})$ is the formal scheme $Spf(\widehat{\cO_{X,x}}[z]/(z^p-f_i)$. It is non regular if and only if ${{\partial}\over{\partial z}}(z^p-f_i)$ and ${{\partial}\over{\partial x_j}}(z^p-f_i)$ belong to the maximal ideal of $\widehat{\cO_{X,x}}[\![z]\!]$  for all $j$. Since ${{\partial}\over{\partial z}}(z^p-f_i)=0$,  and the ideal $({{\partial}\over{\partial x_1}}(z^p-f_i);\dots; {{\partial}\over{\partial x_n}}(z^p-f_i))$ coincides with the ideal $(h_1;\dots; h_n)$ the claim follows. \end{proof}

Suppose that $z\in Z_s$ be a closed singular  point and $x=f_s(z)$. Suppose that the matrix ${{\partial h_i}\over{\partial x_j}}(0)$ is non singular. Then we will say that $z$ is a {\it non degenerate singular point}. One may check that the notion of "non degenerate singular point" depends only on the divisor $div(s)$. in particular it does not depend on the choice of the coordinates around $x$. 

\subsection{Structure and desingularization of $Z_s$ near a non degenerate singular point}
\begin{Claim} Suppose that the point $z\in Z_s$ is a non degenerate singular point and  $x=f_s(z)$. Then there exist formal coordinates $x_1,\dots, x_n$ on $\widehat{\cO_{X,x}}$ for which  $Z_s$ is given by the equations
$$z^p=x_1^2+\dots +x_n^2.$$
\end{Claim}
\begin{proof} Locally, near $x$, the variety $Z_s$ is defined by the equation $z^p=f(x_1,\dots, x_n)$ with
$\det\left({{\partial^2f}\over{\partial x_i\partial x_j}}\right)(x)\neq 0$. Denote by $\cM_{X,x}$ the maximal ideal of $\widehat{\cO_{X,x}}$. Since $z$ is singular, we have that  $f\equiv a_0+{{1}\over{2}}\sum_{i,j}a_{ij}x_ix_j\mod(\cM_{X,x}^3)$ in  $\widehat{\cO_{X,x}}$ with $a_{ij}=a_{ji}$; moreover the symmetric matrix $(a_{ij})$ is non singular because the singularity is non degenerate. The change of variable $z_1:=z-a_0$ gives the new equation $z_1^p=f_1$ for $Z_s$ near $x$, with $f_1(x)=0$ and ${{\partial f_1}\over{\partial x_j}}(x)=0$. To prove the claim it suffices to prove that we can choose formal coordinates $x_1,\dots, x_n$ such that, for every $r$ we have  $f_1(x_1,\dots, x_n)\equiv x_1^2+\dots + x_n^2 \mod (\cM_{X,x}^{r})$. Since we are in characteristic different from two and $\det(a_{ij})\neq 0$, we may suppose that the bilinear form $\sum_{ij} a_{ij}x_ix_j$ is diagonal. Consequently we may suppose by induction on $r$, that  $f_1\equiv x_1^2+\dots+x_n^2\mod(\cM_{X,x}^{r+2})$. Thus $f_1\equiv x_1^2+\dots +x_n^2+\sum_{|I|=r+2}a_Ix^I \mod(\cM_{X,x}^{r+3})$, where $I=(i_1,\dots, i_n)$ is a multi index. Choose a change of variable $x_i=\tilde{x_i}+\sum_{|J|=r+1}b^i_J\tilde{x}^J$. In the new coordinates we have that 
$f_1(\tilde{x_1},\dots,\tilde{x}_n)\equiv\tilde x_1^2+\dots +\tilde x_n^2+2\sum_{i,J}b^i_Jx^J\cdot \tilde x_i+\sum_{|I|=r+2}a_I\tilde x^I \mod(\cM_{X,x}^{r+3})$. Thus a suitable choice of the $b_J^i$'s allows to obtain that $f_1(\tilde{x_1},\dots,\tilde{x}_n)\equiv\tilde x_1^2+\dots +\tilde x_n^2\mod(\cM_{X,x}^{r+3})$.

\end{proof}

We suppose that $Z_s$ has only non degenerate singular points. In this case we remark that the singular points are isolated. We begin by study the disingularization of an affine hyper surface  $Z$ whose equation is 
\begin{equation}\label{localeq}
z^p=x_1^2+\dots +x_n^2.\end{equation}

\

\begin{proposition} The desingularization of the hyper surface \ref{localeq} is obtained by performing $p$ blow ups on isolated singular points. Each of these points is of multiplicity two.
\end{proposition}

\begin{proof} Let $f:\tilde{X}\to\bA^{n+1}$ be the blow up in the point $(0; 0;\dots ;0)$. The local equations of it are given by
$z=vx_i$ and $x_j=u_jx_i$ ($i=1,\dots, n$) or by $x_i=w_iz$. We denote by $E$ the exceptional divisor of $\tilde{X}$.

In the first case the local equation of the strict transform $\tilde{Z}$ of the hyper surface \ref{localeq} is
\begin{equation}
v^{p-2}x_i^{p-2}=1+u_1^2+\dots u_n^2\end{equation}
(the $i$--term is not part of the sum). In this case we remark that the local equation is smooth (because the characteristic of the field is not two).
In the second case the equation of the strict transform is
\begin{equation}
z^{p-2}=w_1^2+\dots +w^{2}_n
\end{equation}
(to simplify notation we put $x_i=w_i$). 
Denote by $\tilde{Z}$ the strict transform of $Z$. We see that $f^\ast(\cO(Z))=\cO(\tilde{Z})(2E)$ thus the multiplicity of  the singular point is two. 
If we blow up again the origin of the last chart we obtain that the equation of the strict transform will be
$z^{p-4}=w_1^2+\dots +w^{2}_n$ and the multiplicity of the singular point is again two.

Thus after ${{p-1}\over{2}}$ blow ups, the local equation of the strict transform is
\begin{equation}
z=w_1^2+\dots +w^{2}_n
\end{equation}
which is smooth and again the multiplicity of the last singular point is two. 
\end{proof}

As a corollary of the proof we obtain the following corollary:

\begin{corollary} \label{cor: desingular} Let $X$ be a smooth variety and $Z\subset X$ be an hyper surface on it.
Suppose that $Z$ has an isolated singular point $P$ and the local formal equation of $Z$ near it is of the form \ref{localeq}. Let $X_1\to X$ be the blow up of $X$ in $P$,  $Z_1$ be the stric transform of $Z$ and $E_1$ be the exceptional divisor of $X_1$. Recursively,  let $X_i\to X_{i-1}$ be the blow up of $X_{i-1}$ in the singular point of $Z_{i-1}$, denote by $Z_i$ the strict transform of $Z_{i-1}$ and by $E_i$ the exceptional divisor of $X_i$.  By abuse of notation, for $j<i$, we denote by $E_j$ the pull back of the divisor $E_j$ to $X_i$. Then:

a) $Z_{(p-1)/2}$ is smooth;

b) if $f:X_{(p-1)/2}\to X$ is the projection, then 
\begin{equation}\label{multiplicity} 
f^\ast(\cO(Z))=\cO(Z_{(p-1)/2})(-\sum_{i=1}^{{p-1}\over{2}}E_i)
\end{equation}
\end{corollary}

\subsection{Inseparable ramified covering of general type}

Suppose now that $X$ is a smooth projective variety of dimension $N$ and $L$ a very ample line bundle on it. Let $s\in H^0(X,L^{np})$  ($n>0$ sufficiently big) a global section such that $\div(s)$ is smooth and $f:Z_s\to X$ the inseparable ramified covering associated to it. We suppose that $Z_s$ has only non degenerate singular points. 

\begin{proposition} In the hypotheses above, Let $\tilde{Z_s}\to Z_s$ be its desingularization  (it exists by corollary \ref{cor: desingular}). If $n$ is sufficiently big then the variety  $\tilde{Z_s}$ is a smooth projective variety of general type. 
\end{proposition}
\begin{proof} The variety $Z_s$ is a divisor inside  the smooth projective variety $Y:=\P(\cO_X\oplus L^{n})$. The variety $\tilde{Z_s}$ is obtained as the strict transform of $Z_s$ in the variety $g:\tilde{Y}\to Y$ obtained by taking successive blow ups at smooth closed points. Denote by $E_{ij}$ the exceptional divisors of $\tilde{Y}$. 

The canonical line bundle of $\tilde{Y}$  will be $g^\ast(K_Y)+N\sum_{ij}E_{ij}=g^\ast(\cO_\P(-2)+L^n+K_X)+N\sum_{ij}E_{ij}$ (we adopt the abuse of notation of corollary \ref{cor: desingular})). 

The class of $Z_s$ in $\Pic(Y)$ will be $\cO_\P(p)+L^{np}$. Thus it is ample on $Y$. The class of $\tilde{Z_s}$ in $\Pic(\tilde{Y})$ will be (cf. \ref{cor: desingular}) $g^\ast(\cO_\P(p)+L^{np})-2\sum_{ij}E_{ij}$. Consequently, by adjonction formula, we have that
\begin{equation}\label{adjunction1}
K_{\tilde{Z_s}}=(K_{\tilde{Y}}+\tilde{Z_s})|_{\tilde{Z_s}}=(g^\ast(\cO_\P(p-2)+L^{np+1} +K_X)+(N-2)\sum_{ij}E_{ij})|_{\tilde{Z_s}}\end{equation}

As soon as $n$ is sufficiently big,  the line bundle $g^\ast(\cO_\P(p-2)+L^{np+1} +K_X)$ is ample on $Z_s$. Thus, for $n$ sufficiently big, the restriction of $g^\ast(K_Y+L^{np})$ is a big and nef line bundle on $\tilde{Z_s}$. The divisor $(N-2)\sum_{ij}E_{ij})$ is effective. Since an effective divisor plus a big and nef is big, the conclusion follows. 
\end{proof}

W show now that, if $s\in H^0(X,L^{np})$ is sufficiently generic and $n$ is sufficiently big, then the associated inseparable ramified covering $Z_s$ has only non degenerate singular points:

\begin{proposition} Suppose that, $L$ is very ample and for every $x\in X$ the restriction map
\begin{equation}\label{surjectivemap}
\alpha:H^0(X,L^{np})\longrightarrow L^{np}\otimes\cO_X/I_x^3
\end{equation}
is surjective ($I_X$ being the ideal sheaf of $x$). Then for $s\in H^0(X,L^{np})$ generic, the inseparable ramified covering $Z_s$ has only non degenerate singular points.

\end{proposition}
\begin{proof} Let $x$ be a point of $X$ and $s\in H^0(X,L^{np})$. If we fix (formal) local coordinates  $z_1,\dots, z_N$ and a local trivialization $f$ of $s$ around $x$, then $\alpha(s)= f(x)+\sum_if_{z_i}(x)z_i+{{1}\over{2}}(\sum_{ij}f_{z_i,z_j}(x)z_iz_j)$. Since the map \ref{surjectivemap} is surjective, for generic $s$, the divisor $\div(s)$ will be smooth and  the quadratic form associated to the matrix $(f_{z_i,z_j})$ will be non degenerate. In this case the associated inseparable ramified covering $Z_s$ will have non degenerate singular points over $x$. We thus see that the set of $s\in H^0(X,L^{np})$ for which the associated inseparable ramified covering $Z_s$ has a singularity which is degenerate at $x$, is a closed set of codimension $N+2$ which we will denote by $S_x$. Indeed the elements of the vector space $\cO_X/I_x^3$ for which the associated quadratic form is degenerate is a closed sub variety of codimension N+2. We will denote again by $S_x$ the image of $S_x$ in $\P(H^0(X,L^{np})$; it will be again a closed set of codimension $N+2$. For a fixed $s$ the set of degenerate singular points of $Z_s$ is a closed set whose projection of $X$ will be denoted by $N_s$

Let $W\subset X\times \P(H^0(X,L^{np})$ be the universal divisor and $N_W$ the corresponding closed set of non degenerate singular points. For every $x\in X$,  the restriction $(N_W)_x$  of $N_W$ to $\{x\}\times\P(H^0(X,L^{np})$ will be  $S_x$. Thus the dimension of $N_W$ is $h^0(X,L^{np})-1-(N+2)+N=h^0(X,L^{np})-3$. This means that $N_W$ do not dominate $\P(H^0(X,L^{np})$. Consequently, for generic $s\in \P(H^0(X,L^{np})$,  the corresponding $Z_s$ has only non degenerate singular points. \end{proof}

\subsection{Non isotrivial inseparable ramified coverings} 

Suppose now that $K$ is a function field of positive characteristic $p>0$. Suppose that $X$ is a variety defined over the base field $k$ and $L$ is an ample line bundle over it. Let $s\in H^0(X,L^{np})$ be a smooth section and $g:Z_s\to X$ the associated inseparable  ramified covering. Denote by $Y_s$ the divisor $\div(s)$. We are going to relate the Kodaira--Spencer class of $Y_s$ with the Kodaira--Spencer class of $\tilde{Z_s}$:

$\tilde{Z_s}$ is a divisor in a blow up of the projective bundle $\P:=\P(\cO_X\oplus L^n)$. Let $\cO_\P(1)$ be the tautological line bundle of $\P$. 

We fix formal coordinates $x_1\dots, x_n$ of $X$  and a local equation $f=0$ of $s$ around a point of $Y_s$. Thus a local equation for $Z_s$ is $z^p=f$. 

a) The sheaf of differentials $\Omega^1_{Y_s/K}$ is given by $(\oplus_{i=1}^n\cO_{Y_s}dx_i)/df$.

b) The sheaf of differentials $\Omega^1_{Z_s/K}$ is given by $(\cO_{Z_s}dz\oplus_{i=1}^n\cO_{Z_s}dx_i)/df$ (observe that the relations do not contain $dz$).

c) Let $W_s$ be the divisor pre image of $Y_s$ in $Z_s$. Its local equation in $Z_s$ is $f=0$. Denote by $g_s:W_s\to Y_s$ the restriction of $g$ to $W_s$. From (a) and (b) above we see that the natural map
\begin{equation}
(\Omega^1_{Z_s/K})|_{W_s}\longrightarrow\Omega^1_{W_s/K}
\end{equation}
is an isomorphism. 

d) Locally the sheaf  $\cO_{Y_s}$ is $A/(f)$ and the local sheaf of $W_s$ is $(A/(f)[z])/(z^p)$. Thus the natural inclusion $\cO_{Y_s}\to g_{s,\ast}(\cO_{W_s})$ is split  (remark that no singular point of $Z_s$ is located on $W_s$).  This, together with (c) above implies that the natural map

\begin{equation}
\alpha_{Y_s}: H^1(Y_s; (\Omega^1_{Y_s/K})^\vee)\longrightarrow H^1(W_s, g_s^\ast(\Omega^1_{Y_s/K})^\vee).
\end{equation}
is an inclusion.

e) Again, by the descriptions in (a), (b) and (c) above we get an exact sequence
\begin{equation}\label{differentialsofX}
0\to f_s^\ast(\Omega^1_{Y_s/K})\longrightarrow\Omega^1_{W_s/K}\longrightarrow \cO(1)\otimes L^{np}\to 0.
\end{equation}
This exact sequence, together with (d) give rise to {\it an inclusion}
\begin{equation}
\alpha_{Y_s}:H^1(Y_s,(\Omega^1_{Y_s/K})^\vee)\longrightarrow H^1(W_s;(\Omega^1_{W_s/K})^\vee).
\end{equation}

f) From the descriptions above and taking duals we get natural maps
\begin{equation}
H^1(\tilde{Z_s};(\Omega^1_{\tilde{Z_s}/K})^\vee)\buildrel{\alpha_{Z_s}}\over\longrightarrow H^1(W_s;(\Omega^1_{W_s/K})^\vee)\buildrel{\alpha_{X_s}}\over\longleftarrow H^1(Y_s;(\Omega^1_{Y_s/K})^\vee).
\end{equation}
A simple (but tedious) diagram chasing gives $\alpha_{Z_s}(KS(\tilde{Z_s}))=\alpha_{X_s}(KS(Y_s))$. 

Thus we deduce the following statement:

\begin{proposition} \label{isotrivial1} The non vanishing of the  of Kodaira Spencer class of $Y_s$ implies the non vanishing of the Kodaira Spencer class of   the variety $\tilde{Z_s}$.   
\end{proposition}

From the constructions above we get the following theorem:

\begin{theorem} Suppose that $X$ is a smooth projective surface defined over the base field $k$ and $L$ is a sufficiently ample line bundle over it. Let $X_K$ be the base change of it to $K$ and $s\in H^0(X_K;L^{np})$ be a non isotrivial smooth divisor. Then the associated inseparable ramified covering $Z_s$ is not birational to an isotrivial surface. 
\end{theorem} 

\begin{proof} From proposition \ref{isotrivial1} and fact \ref{isotrivial2} we get that $\tilde{Z_s}$ is not isotrivial. Formula \ref{adjunction1}  computes the canonical line bundle of $\tilde{Z_s}$. Thus we get that $\tilde{Z_s}$ is of general type and minimal. Since two minimal surfaces of general type are isomorphic if and only if they are birationally equivalent, the proposition follows. 
\end{proof}

\begin{remark} In higher dimension we can only conclude that the variety $\tilde{Z_s}$ is not defined over $k$. It is possible that a finer study, using MMP, may allow to deduce that $\tilde{Z_s}$ is not birational to a variety defined over $k$. 
\end{remark}

\subsection{Inseparable ramified coverings and Frobenius} We recall here some standard facts about the Frobenius morphism of a variety. Let $\overline{K}$ be the algebraic closure of $K$.  If $X$ is a variety over $\overline{K}$, we denote by $F_X:X\to X$ the Frobenius morphism (it is the identity on the topological space and $f\to f^p$ on functions). The Frobenius morphism fits inside a diagram
\begin{equation}
\xymatrix{X\ar[r]^{F_X^g}\ar[rd]&X^{(1)}\ar[r]\ar[d]&X\ar[d]\\
&\Spec(\overline{K})\ar[r]^{F_K}&\Spec(\overline{K})\\}
\end{equation}
where, $F_K$ is the Frobenius morphism of $K$, the square on the right is cartesian and $F^g_X$ is a $\overline{K}$ morphism called {\it the geometric Frobenius}. Suppose now that $X$ is a smooth projective $\overline{K}$ variety and $\overline{K}(X)$ is the field of rational functions of it. If $\overline{K}(X^{(1)})=\overline{K}(x_1,\dots., x_r)$ then the field morphism associated to $F^g_X$ is $ \overline{K}(x_1,\dots., x_r)\buildrel{F^g_X}\over{\longrightarrow}\overline{K}(x_1,\dots., x_r)[T_1.\dots, T_r]/_{(T_1^p-x_1, \dots, T^p_r-x^r)}= \overline {K}(X)$.

Suppose now that $f:\tilde{Z_s}\to X^{(1)}$ is an inseparable ramified morphism associated to a global section of a line bundle over $X^{(1)}$. Then the field of rational functions of $\tilde{Z_s}$ is
$\overline{K}(\tilde{Z_s})=\overline{K}(X)[z]/(z^p-h)$ where $h$ is a suitable rational function over $X^{(1)}$. Write $h=\sum a_Ix^I$ where $I$ is a multiindex $(i_1,\dots, i_r)$, $a_I\in \overline {K}$ and $x^I:=x_1^{i_1}x_2^{i_2}\dots x_r^{i_r}$. For every $I$ let $b_I\in  \overline {K}$ such that $b_I^p=a_I$. Thus we obtain an inclusion  $\overline{K}(\tilde{Z_s})\hookrightarrow  \overline {K}(X)$ by sending $z$ to $\sum b_IT^I$.

Consequently we get the following:
\begin{proposition} Let $X$ be a smooth projective variety defined over $K$ and $f:\tilde{Z_s}\to X^{(1)}$ be an inseparable ramified covering associated to a section of a suitable line bundle on it. Then there exists a finite extension $K'$ of $K$, a blow up $\tilde {X}\to X$ and a dominant (inseparable) morphism
$h:\tilde {X}\to \tilde{Z_s}^{(1)}$.
\end{proposition}

\subsection{ Inseparable ramified coverings and arithmetic over function fields} Let $K$ be a function field of one variable over an algebraically closed field $k$ of characteristic $p>0$. From the construction above we see that, given a smooth projective surface $X_0$ defined over the base field $k$, we can construct surfaces $\tilde{Z_s}^{(1)}$ over $K$ such that:

a)  $\tilde{Z_s}^{(1)}$ is smooth, projective and of general type.

b)  $\tilde{Z_s}^{(1)}$ is not birational to an isotrivial surface.

c) There is a blow up $\tilde{X_0}$ of $X_0\otimes_kK$ and a dominant (non separable) morphism $f:\tilde{X}\otimes_kK\to\tilde{Z_s}^{(1)}$.

To prove (c) just remark that if $Y$ is a variety, then $Y$ is defined over $k$ if and only if $Y^{(1)}$ is.

We list now two important consequences of this:

{\it 1) The image by $f$ of each $k$ point of $\tilde{X}_0\otimes_kK$  is a $K$-rational point of $\tilde{Z_s}^{(1)}$.}

{\it Consequence}: The set of $K$--rational points of of bounded height  $\tilde{Z_s}^{(1)}$ is Zariski dense. 

{\it 2) Suppose that $X_0=\P_2$. Then every form of Vojta inequality fails  for $\tilde{Z_s}^{(1)}$.}

Let's give some details about  the proof of consequence (2):   In this case a model of $X_0$ over $B$ is $\P_2\times B$. Fix a normal projective model $\overline Z\to B$ of $\tilde{Z_s}^{(1)}$. Then (up to an extension of $K$ if necessary), we can find a proper closed set $W\subset \P_2\times B$ of codimension at least two  such that, if $X_1\to \P_2\times B$ is the blow up of it, we have a dominant map $h:X_1\to \overline Z$. The lemma below tells us that we can find a Zariski dense set of points $p\in X_1(K)$ having constant discriminant $d_p$ and unbounded height with respect to a (any) ample line bundle. Indeed the pre image in $\P_2\times B$ of almost every line  in $\P_2$ will intersect $W$ in only finitely many points. 

The image via $h$ of these points is a set of points which violates Vojta inequality.

\begin{lemma} Let $B$ be a smooth projective curve and $W$ be a finite set of points in $B\times\P_1$ then there are infinitely many sections $g:B\to B\times\P_1$ which do not intersect $W$.
\end{lemma}
\begin{proof} It suffices to observe that we can find a line bundle $L$ on $B$ such that $M:=p_{\P_1}^\ast(\cO_\P(1))\otimes p_B^\ast(L)$ is very ample on $B\times\P_1$. Every smooth global section of $M$ which avoids $W$ satisfy the conclusion of the theorem. 
\end{proof}

Consequences (1) and (2) above show that a "naive" version of Lang and Vojta conjectures are definitely false in positive characteristic. Once again this is due to the existence of isotrivial varieties (which in positive characteristic are even more mysterious then in characteristic zero).

\end{document}